\documentclass[12pt]{amsart}

\usepackage{amssymb}
\usepackage{amsmath}
\usepackage{amsfonts}
\usepackage{graphicx}
\usepackage{amsthm}
\usepackage{enumerate}
\usepackage[mathscr]{eucal}
\usepackage{mathrsfs}
\usepackage{verbatim}
\usepackage{booktabs}
\usepackage{amssymb}

\makeatletter
\@namedef{subjclassname@2010}{%
	\textup{2010} Mathematics Subject Classification}
\makeatother

\numberwithin{equation}{section}

\theoremstyle{plain}
\newtheorem{theorem}{Theorem}[section]
\newtheorem{lemma}[theorem]{Lemma}

\theoremstyle{plain}

\numberwithin{equation}{section}

\theoremstyle{remark}
\newtheorem{remark}[theorem]{Remark}

\DeclareMathOperator{\supp}{supp}
\DeclareMathOperator{\dist}{dist}
\DeclareMathOperator{\vol}{vol}

\begin{document}

\date{}
\title
[variance of lattice point counting]{variance of lattice point counting in some special shells in $\mathbb{R}^d$ }
\author[t. jiang]{tao jiang}

\address{Tao Jiang\\Department of Mathematics\\
University of Science and Technology of China\\
Hefei, Anhui Province 230026, People's Republic of China}

\email{jt1023@mail.ustc.edu.cn}

\thanks{}

\begin{abstract}
We study the variance of the random variable that counts the number of lattice points in some shells generated by a special class of finite type domains in $\mathbb R^d$. The proof relies on estimates of the Fourier transform of indicator functions of convex domains.
\end{abstract}

\subjclass[2010]{11P21, 42B10, 60D05}
\keywords{Lattice point, finite type, Fourier transform.}

\maketitle


\section{Introduction}
The classical lattice point problem associated with a compact convex domain $\mathcal{B}\subseteq \mathbb{R}^d$ is about counting the number of lattice points $\mathbb{Z}^d$
in the enlarged domain $r\mathcal{B}$ and the main problem is to study the remainder
\begin{equation*}
R_{\mathcal{B}}(r)=\#\left(r\mathcal{B}\cap\mathbb{Z}^d\right)-\vol(\mathcal{B})r^d
\end{equation*}
for $r\geq 1$. If $\mathcal B$ is the unit disk centered at the origin, it is the well-known Gauss circle problem and it is conjectured that $R_{\mathcal{B}}(r)=O_{\epsilon}(r^{1/2+\epsilon})$ for any $\epsilon>0$. For the Gauss circle problem, Huxley \cite{Huxley} in 2003 proved that $R_{\mathcal{B}}(r)=O(r^{131/208+\epsilon})$ and Bourgain and Watt \cite{BW2017} obtained $R_{\mathcal{B}}(r)=O(r^{517/824+\epsilon})$ in 2017. Many authors made efforts to study general domains under different curvature assumptions of the boundary $\partial \mathcal B$. We refer interested readers to Ivi\'{c}, Kr\"{a}tzel, K\"{u}hleitner and Nowak \cite{survey2006} and Nowak \cite{survey2014} which gave excellent overview of the development of this problem.

 Assume $\mathcal B\subseteq \mathbb{R}^d$ is a convex domain which contains in its interior the origin and has smooth boundary with nonzero Gaussian curvature. Denote by
 \begin{equation*}
 \mathcal B(r,t)=(r+t/2)\mathcal B\setminus(r-t/2)\mathcal B
 \end{equation*}
 the shell generated by the domain $\mathcal B$ and
 \begin{equation*}
 N_{\mathcal B(r,t)}(u)=\#\{n\in \mathbb Z^d:n\in(\mathcal B(r,t)-u)\}
  \end{equation*}
  the random variable that counts the number of lattice points in the shifted shell $\mathcal B(r,t)-u$. Then
Colzani, Gariboldi and Gigante \cite{CGG} proved that the variance of $N_{\mathcal B(r,t)}(u)$ is asymptotic to $\vol(\mathcal B(r,t))$ as $r$ goes to infinity. Here $0<t\le r^{-\alpha}$ for every $\alpha>(d-1)/(d+1)$ and $u$ is uniformly distributed in $\mathbb T^d$.

In this paper, we study the variance of the random variable
\begin{equation*}
N_{\mathcal D(r,t)}(u)=\#\left\{n\in\mathbb Z^d:n\in(\mathcal D(r,t)-u)\right\},
\end{equation*}
 where the shell
\begin{equation}\label{between two enlarged domains}
 \mathcal D(r,t)=(r+t/2)\mathcal{D} \backslash(r-t/2)\mathcal{D}
\end{equation}
is generated by convex domain
\begin{equation}\label{modeldomain}
\mathcal{D}=\left\{x\in\mathbb{R}^{d}  :  \sum _{p=0}^{n-1}\left(\sum _{l=1+d_{p}}^{d_{p+1}} x_{l}^{\omega_{l}}
\right)^{m_{p+1}}\le 1\right\}.
\end{equation}
Here $n$, $p$, $m_{p+1}, d_{p+1}\in \mathbb{N}\cup \{0\}$\footnote{In this paper we use $\mathbb N=\{1,2,\ldots\}$.}, $0=d_0< d_1< \cdots <d_{n-1}<d_n=d$, $d\ge 3$ and $4\le\omega_l\in 2\mathbb{N}$. The authors in \cite{CGG} pointed out that the curvature assumption is important and gave some examples to show that if the curvature of some points on $\partial\mathcal B$ vanishes, the variance may be much larger than the volume of the shell. Notice that the examples they gave are some shells generated by polyhedra. We are interested in the situation between these two cases where the shells are generated by finite type domains\footnote{A compact convex domain $\mathcal B$ is a finite type domain if for any point $x\in \partial \mathcal B$, any one dimensional tangent line of $\partial \mathcal B$ at $x$ makes finite order of contact with $\partial \mathcal B$. See \cite[p. 351]{BNW} for an analytic definition.}. Since it is difficult to study the most general finite type domain, we start with some special cases. Notice that the following special cases are considered in the classical lattice point problem. The super spheres
\begin{equation}\label{supersphere}
\left\{x\in \mathbb R^d:|x_1|^{\omega}+\cdots+|x_d|^{\omega}\le 1\right\}
\end{equation}
were considered in Randol \cite{randol} for even $\omega\ge 3$ and Kr\"{a}tzel \cite{kratzel_odd} for odd $\omega\ge 3$. Kr\"{a}tzel \cite{kratzel_2002} and Kr\"{a}tzel and Nowak \cite{K-N-2008,K-N-2011} studied the convex domain
\begin{equation}\label{exampledoamin}
\left\{x\in \mathbb R^3: |x_1|^{mk}+\left(|x_2|^k+|x_3|^k\right)^m\le 1\right\}
\end{equation}
with certain assumptions on the reals $k$ and $m$. Motivated by these examples, we consider the shells generated by the domain $\mathcal D$ defined by \eqref{modeldomain}. The domain $\mathcal D$ is formally an extension of \eqref{supersphere} and \eqref{exampledoamin} and the Gaussian curvature of the intersection of $\partial \mathcal D$ and the coordinate planes is zero.

In order to state our main theorem, we first define some notations. For the constants $d_i$'s and $m_j$'s appearing in \eqref{modeldomain}, given any $1\leq l\leq d$, there exists a unique $0\leq p(l)\leq n-1$ such that $1+d_{p(l)}\leq l\leq d_{p(l)+1}$. For $1\le q\le d$, define
\begin{equation}\label{def-mjl}
m_{q,l}=\bigg\{ \begin{array}{ll}
1 & \quad\textrm{if $1\leq l\leq d$ and $p(l)=p(q)$},\\
m_{p(l)+1} & \quad\textrm{if $1\leq l\leq d$ and  $p(l)\ne p(q)$}.
\end{array}
\end{equation}
For $1\le q\le d$ and $1\le j\le d-1$, let
\begin{equation}\label{alphaqj}
\alpha_{q,j}=\!\!\max_{\substack{S\in P_j\{1,2,\ldots,d\}\\S\owns q}}\left\{d-j-\sum_{l=1,l\notin S}^{d}\frac{2}{m_{q,l}\omega_l},\frac{d-j-\sum_{l=1,l\notin S}^{d}2/(m_{q,l}\omega_l)}{\sum_{l=1,l\notin S}^{d}2/(m_{q,l}\omega_l)}\right\}
\end{equation}
and
\begin{equation}\label{alphaj}
\alpha_j=\max\{\alpha_{1,j},\alpha_{2,j},\ldots,\alpha_{d,j}\}.
\end{equation}
Here $P_j\{1,2,\ldots,d\}$ denotes the collection of all subsets of $\{1,2,\ldots,d\}$ having $j$ elements. It seems that $\alpha_{q,j}$'s are complicated. These constants appear in the proof of Lemma \ref{lemmaZd(j)} below. To help readers understand them better, Remark \ref{remark} gives some examples of the domain $\mathcal D$ and the corresponding constants $m_{q,l}$'s and $\alpha_{q,j}$'s.

It is easy to verify that the expectation of $N_{\mathcal D(r,t)}(u)$ is the volume of $\mathcal D(r,t)$. For the variance we have
\begin{theorem}\label{variance}
For the domain $\mathcal D(r,t)$ defined by \eqref{between two enlarged domains}, if $1< r<+\infty$, $0<t\le Cr^{-\alpha}$ for some constant $C>0$ and $\alpha>\max\{\alpha_1,\alpha_2,\ldots,\alpha_{d-1}\}$, then there is a constant $\beta>0$ such that
\begin{equation}\label{varianceequation}
\int_{\mathbb{T}^d}\left|N_{\mathcal D(r,t)}(u)-\vol(\mathcal D(r,t))\right|^2\,\mathrm du=\vol(\mathcal D(r,t))
\left(1+O\left(t^\beta\right)\right),
\end{equation}
where the implicit constant is independent of $r$ and $t$.
\end{theorem}

\begin{remark}\label{remark}
For the domain $\mathcal D$ defined by \eqref{modeldomain},
\begin{enumerate}[(1)]
\item if $\omega_1=\cdots=\omega_d$ and $m_1=\cdots=m_n=1$, then
\begin{equation}\label{exampledomain1}
\mathcal D=\left\{x\in\mathbb{R}^{d}  :  x_1^{\omega_1}+\cdots+x_d^{\omega_1}\le 1\right\}.
\end{equation}
Given any $1\le q,l\le d$ and $1\le j\le d-1$, we have $m_{q,l}=1$ and
\begin{equation*}
\alpha_{q,j}=\frac{(d-j)(1-2/\omega_1)}{\min\{1,2(d-j)/\omega_1\}}.
\end{equation*}
\item if $d=3$, $n=2$, $d_1=2$, $\omega_1=4,\omega_2=6,\omega_3=8$ and $m_1=5,m_2=1$, then
 \begin{equation*}\label{exampledomain2}
\mathcal D=\left\{x\in\mathbb{R}^{3}: \left(x_1^{4}+x_2^{6}\right)^{5}+x_3^8\le 1\right\}.
\end{equation*}
Given any $q=1,2$ and $1\le l\le3$, we have $m_{q,l}=1$,
\begin{align*}
 m_{3,l}=\left\{\begin{array}{ll}
5&\text{if $l=1,2,$}\\
1&\text{if $l=3$},
\end{array}
\right.
\end{align*}
$\alpha_{1,1}=17/7$, $\alpha_{1,2}=3$, $\alpha_{2,1}=5/3$, $\alpha_{2,2}=3$, $\alpha_{3,1}=11$ and $\alpha_{3,2}=14$.
\end{enumerate}
\end{remark}

\begin{remark}
We can choose $t=Cr^{-(d-1)}$ for some constant $C>0$ if
\begin{equation}\label{property}
\max\{\alpha_1,\ldots,\alpha_{d-1}\}< d-1\footnote{This is the case for the domain defined by \eqref{exampledomain1} with $\omega_1<2d$. }.
\end{equation}
Then we have the property that the variance converges to the bounded nonzero constant $\vol(\mathcal D)dC$  as $r$ goes to infinity. This is consistent with the result in Cheng, Lebowitz, Major \cite{CLM} in $\mathbb R^2$.
There they take $t=2c/r$ and the variance tends to $4c$ times the area of the domain.

\end{remark}

 \begin{remark}
It is necessary to assume a range of $\alpha$ in the theorem. An example is given in the last part of Appendix to show that if $\alpha$ takes some value less than $\max\{\alpha_1,\ldots,\alpha_{d-1}\}$, the variance is much larger than the volume of $\mathcal D(r,t)$ when $r$ is large enough. But the range of $\alpha$ in Theorem \ref{variance} may not be sharp.
 \end{remark}

\begin{remark}
This paper extends the result in \cite{CGG} from the case of non-vanishing Gaussian curvature to the case where there are boundary points with zero Gaussian curvature. The difficulty resolved in this paper is how to handle these points. Since the problem will be reduced to estimating a sum involving the Fourier transform of the indicator function of $\mathcal D(r,t)$, we shall divide $\partial\mathcal D$ into different parts according to the Gaussian curvature. For the part with zero Gaussian curvature, we apply Theorem 2.2 in \cite[p. 74]{model2} to sum the series and this requires more computation.
\end{remark}

{\it Notations:}  For any convex domain $\mathcal B$, let $n(x)$ be the unit outer normal of the point $x$ on $ \partial \mathcal B$ and $x(\xi)$ be the point on $\partial \mathcal B$ with nonzero outer normal $\xi$. Denote by $K(x)$ the Gaussian curvature of $x$ on $\partial\mathcal B$. Let $\nabla f$ be the gradient and $\nabla^2 f$ be the Hessian matrix of function $f$. $\textbf A^T$ is the transpose of matrix $\mathbf A$.  $C^{\infty}(U)$ is the space of all smooth functions on the open set $U$.
For functions $f$ and $g$, the notation $f\lesssim g$ or $f=O(g)$ means $|f|\le Cg$ for some constant $C>0$ and $f\gtrsim g$ equals $g\lesssim f$. If $f$ and $g$ are non-negative functions, $f\asymp g$ means there are some positive constants $c'$ and $C'$ such that $c'f\le g\le C'f$. Let $\mathbb T^d$ be the unit cube $\{x\in\mathbb R^d: -1/2\le x_i<1/2\,\text{ for}\, 1\le i\le d\}$.

Throughout this paper the notation $\mathcal D$ and $\mathcal D(r,t)$ are the domains defined by \eqref{between two enlarged domains} and \eqref{modeldomain} and $m_{q,l}$ is the constant defined by \eqref{def-mjl}. Denote by $\varepsilon_0$ a small constant such that for any nonzero $\xi\in \mathbb R^d$, there is a $1\le q\le d$ satisfying $|\xi_q|/|\xi|\ge \varepsilon_0$.



\section{some estimates}
Before we prove Theorem \ref{variance}, we give some results that will be used later: Lemma \ref{lemma1} and \ref{relation} are some simple results on the Gaussian curvature. The proofs are basic calculations under the local coordinates so we put them in Appendix. Lemma \ref{asymp}--\ref{A and B} are some asymptotic expansions and estimates  related to the Fourier transform of the indicator functions of convex domain and $
\mathcal D(r,t)$. Lemma \ref{two} is an estimate of an oscillatory integral. Given any $x\in \mathbb R^d$, denote by $\hat x_q\in \mathbb R^{d-1}$ the $(d-1)$-dimensional vector obtained by deleting the $q^{th}$ component from $x$. For example, $\hat x_1=(x_2,\ldots,x_d)$. Then we have

\begin{lemma}\label{lemma1}
Let $\mathcal{B}\subseteq \mathbb{R}^d$ ($d\ge 3$) be a
compact convex domain. Assume that given any $x\in\partial\mathcal{B}$, $x$ has a small neighborhood on $\partial \mathcal B$ that can be given as a graph $x_q=f(\hat x_q)$ with $f\in C^{\infty}(\mathbb R^{d-1})$ for some $1\le q\le d$. Then we have the following results.
\begin{enumerate}[(1)]
\item The Gaussian curvature of $\partial \mathcal B$ at $x$ is
\begin{align}\label{Gaussiancurvature}
K(x)=\frac{|\det\left(\nabla^2f(\hat x_q)\right)|}{\left(1+|\nabla f(\hat x_q)|^2\right)^{(d+1)/2}}.
\end{align}
\item If $K(x)\ne0$, the Gauss map $n$ has a smooth inverse in a small neighborhood of $ n(x)$, namely, given any $\xi/|\xi|$ in this neighborhood, there is a unique point
$x(\xi)\in \partial \mathcal B$ such that $n(x(\xi))=\xi/|\xi|$. Moreover,
\begin{equation}\label{inverse of Gauss map}
\frac{\partial x_i(\xi)}{\partial \xi_j}=-\frac{A_{ji}}{\xi_q\det\left(\nabla^2f(\hat x_q)\right)},
\end{equation}
where $A_{ji}$ is the cofactor of $\nabla^2 f(\hat x_q)$ of the entry $\frac{\partial^2 f(\hat x_q)}{\partial{x_i}\partial{x_j}}$ for $1\le i,j\le d$ and $i,j\ne q$.
\begin{equation}\label{inverse of Gauss map q}
\frac{\partial x_i(\xi)}{\partial \xi_q}=-\frac{\det (\textbf A_{qi})}{\xi_q\det\left(\nabla^2f(\hat x_q)\right)},
\end{equation}
where $1\le i\ne q\le d$ and if $i<q$, $\textbf A_{qi}$ denotes the matrix obtained by replacing the $i^{th}$ column of $\nabla^2 f(\hat x_q)$ with $(\nabla f(\hat x_q))^{T}$;  if $i>q$, $\textbf A_{qi}$ denotes the matrix obtained by replacing the $(i-1)^{th}$ column of $\nabla^2 f(\hat x_q)$ with $(\nabla f(\hat x_q))^{T}$.
\end{enumerate}
\end{lemma}
\begin{proof}
See Appendix 1.
\end{proof}
If $|\xi_q|\asymp |\xi|$, from \eqref{inverse of Gauss map} and \eqref{inverse of Gauss map q} we have
\begin{equation*}
\frac{\partial x_i(\xi)}{\partial \xi_j}\lesssim |\xi|^{-1}(K(x(\xi)))^{-1}.
\end{equation*}
Moreover, let $\nu=(\nu_1,\ldots,\nu_d)$ be a multi-index. Then
\begin{equation}\label{partial x(xi)}
\frac{\partial^{|\nu|} x_i(\xi)}{\partial \xi^{\nu}}\lesssim |\xi|^{-|\nu|}(K(x(\xi)))^{-1-2(|\nu|-1)},
\end{equation}
Combing \eqref{Gaussiancurvature} and \eqref{partial x(xi)} yields
\begin{equation}
 \frac{\partial^{|\nu|} (K(x(\xi)))^{-\gamma}}{\partial \xi^{\nu}}\lesssim |\xi|^{-|\nu|}(K(x(\xi)))^{-\gamma-2|\nu|}\label{partial K}
\end{equation}
for some $\gamma>0$.

\begin{lemma}\label{relation}
Let $1\le q\le d$ be an integer. For the domain $\mathcal D$, given any nonzero $\xi\in \mathbb R^d$ with $|\xi_q|/|\xi|\ge \varepsilon_0$, we have
\begin{equation}\label{relation between point and direction}
\prod _{i=1,i\ne q}^{d}\left(|\xi_i|/|\xi|\right)^{\frac{m_{q,i}\omega_i-2}{m_{q,i}\omega_i-1}}
\lesssim K(x(\xi))\lesssim\prod _{i=1,i\ne q}^{d}\left(|\xi_i|/|\xi|\right)^{\frac{\omega_i-2}{\omega_i-1}},
\end{equation}
where the implicit constants only depend on the domain $\mathcal D$ and $\varepsilon_0$.
\end{lemma}
\begin{proof}
See Appendix 2.
\end{proof}
If there is an integer $1\le i\le d$ such that $\xi_i/|\xi|=0$, from \eqref{relation between point and direction} we know $K(x(\xi))=0$.  For $1\le j\le d$, define
\begin{equation}\label{Zd(j)}
\mathbb{Z}^d(j)=\cup_{S\in P_j\{1,2,\ldots,d\}}\{n\in \mathbb Z^d: n_i\ne 0\,\text{if}\, i\in S,\,\text{otherwise}\, n_i=0\},
\end{equation}
where $P_j\{1,2,\ldots,d\}$ denotes the collection of all subsets of $\{1,2,\ldots,d\}$ having $j$ elements. Given any $n\in \mathbb Z^d(j)$,  it is easy to verify that $K(x(n))=0$ if $1\le j\le d-1$ and $K(x(n))\ne 0$ if $j=d$.

\begin{lemma} \label{asymp}
Assume $\mathcal{B}\subseteq \mathbb{R}^d\,(d\ge 3)$ is a compact convex domain and its boundary is a smooth hypersurface. Then given any $N\in \mathbb N$ and nonzero $\xi \in \mathbb R^d$ with $K(x(\pm\xi))\ne 0$, we have
\begin{equation}\label{expansion of ft}
\widehat{\chi}_{\mathcal{B}}(\xi)=a(\xi)|\xi|^{-(d+1)/2}+E(\xi),
\end{equation}
where
\begin{align*}
a(\xi)&=(2\pi i)^{-1}e^{-2\pi ix(-\xi)\cdot\xi-\pi i(d-1)/4}(K(x(-\xi)))^{-1/2}\\
&\quad-(2\pi i)^{-1}e^{-2\pi ix(\xi)\cdot\xi+\pi i(d-1)/4} (K(x(\xi)))^{-1/2}
\end{align*}
and
\begin{align}\label{est nabla E}
|E(\xi)|+|\nabla E(\xi)|\lesssim |\xi|^{-(d+3)/2}\delta^{-(d+5)/2}+|\xi|^{-N-1}\delta^{-4N}.
\end{align}
Here $\delta=\min\left\{K(x(\xi)), K(x(-\xi))\right\}$ and the implicit
constant only depends on $\mathcal{B}$ and $N$.
\end{lemma}
\begin{proof}
We only prove the estimate of $\nabla E(\xi)$, see Corollary 4.3 in Guo \cite[p. 423]{highdim} for the first part of this lemma. Combining \eqref{partial K}, \eqref{expansion of ft} and the fact $\nabla (x(\xi)\cdot \xi)=x(\xi)$ (see Corollary 1.7.3 in \cite[p. 47]{supportfunction}) yields
\begin{equation}
\begin{split}\label{nabla E}
\nabla E(\xi)&=\nabla\widehat{\chi}_{\mathcal{B}}(\xi)+2^{-1}(d+1)|\xi|^{-(d+5)/2}\xi a(\xi)
-|\xi|^{-(d+1)/2}\nabla a(\xi),\\
&=\nabla\widehat{\chi}_{\mathcal{B}}(\xi)-|\xi|^{-(d+1)/2}\nabla a(\xi)+O\left( |\xi|^{-(d+3)/2}\delta^{-1/2}\right)\\
&=\nabla\widehat{\chi}_{\mathcal{B}}(\xi)+|\xi|^{-(d+1)/2}e^{-2\pi ix(-\xi)\cdot\xi-\pi i(d-1)/4}x(\!-\xi)(K(x(-\xi)))^{-1/2}\\
&\quad-|\xi|^{-(d+1)/2}e^{-2\pi ix(\xi)\cdot\xi+\pi i(d-1)/4}x(\xi)(K(x(\xi)))^{-1/2}\\
&\quad+O\left(|\xi|^{-(d+3)/2}\delta^{-5/2}\right).
\end{split}
\end{equation}

By the divergence theorem we have
\begin{equation}
\begin{split}\label{nabla chiB}
\nabla\widehat{\chi}_{\mathcal{B}}(\xi)&=|\xi|^{-2}\int_{\partial{\mathcal{B}}}xe^{-2\pi i\xi\cdot x}\xi\cdot n(x)\,\mathrm d\sigma(x)-|\xi|^{-2}\xi\int_{\mathcal B}e^{-2\pi i\xi\cdot x}\,\mathrm dx\\
&=|\xi|^{-2}\int_{\partial{\mathcal{B}}}xe^{-2\pi i\xi\cdot x}\xi\cdot n(x)\,\mathrm d\sigma(x)+O\left(|\xi|^{-1}|\widehat{\chi}_{\mathcal{B}}(\xi)|\right)\\
&=|\xi|^{-(d+1)/2}e^{-2\pi ix(\xi)\cdot\xi+\pi i(d-1)/4}x(\xi)(K(x(\xi)))^{-1/2}\\
&\quad -|\xi|^{-(d+1)/2}e^{-2\pi ix(-\xi)\cdot\xi-\pi i(d-1)/4}x(-\xi)(K(x(-\xi)))^{-1/2}\\
&\quad +O\left(|\xi|^{-(d+3)/2}\delta^{-(d+5)/2}+|\xi|^{-N-1}\delta^{-4N}\right),
\end{split}
\end{equation}
where $\mathrm d\sigma$ is the induced surface measure on the boundary and the last equality follows from the proof of Theorem 4.2 in \cite[p. 422-423]{highdim}.
Combining \eqref{nabla E} and \eqref{nabla chiB} yields the desired bounds of $\nabla E(\xi)$. This finishes the proof.
\end{proof}

Notice that the symmetry of $\mathcal D$ implies $x(-\xi)\!=-x(\xi)$ and $K(x(\xi))=K(x(-\xi))$. Hence if $K(x(\xi))\ne 0$, applying Lemma \ref{asymp} yields
\begin{equation*}
\widehat{\chi}_{\mathcal{D}}(\xi)=a(\xi)|\xi|^{-(d+1)/2}+E(\xi),
\end{equation*}
where
\begin{equation}\label{a(xi)}
a(\xi)=\pi^{-1}(K(x(\xi)))^{-1/2}\sin(2\pi x(\xi)\cdot \xi-(d-1)\pi/4).
\end{equation}
Meanwhile, by the result in Svensson \cite[p. 19]{Svensson} we have
\begin{equation}\label{svensson}
\widehat{\chi}_{\mathcal{D}}(\xi)\lesssim |\xi|^{-(d+1)/2}(K(x(\xi)))^{-1/2}.
 \end{equation}

Let $T_{x}$ be the tangent plane of $\partial \mathcal D$ at $x$,
\begin{equation*}
\tilde{B}(x,\lambda)=\left\{y\in \partial \mathcal D: \dist\left(y, T_{x}\right) < \lambda\right\}
 \end{equation*}
 be a ``cap" cut off from $\partial\mathcal D$ by a plane parallel to $T_{x}$ at distance $\lambda$, and $\sigma \left(\tilde{B}(x,\lambda)\right)$ be the surface measure of $\tilde{B}(x,\lambda)$. Then

\begin{lemma}\label{lemmacap}
Let $1\leq q\le d$ be an integer. For the domain $\mathcal{D}$, given any nonzero $\xi\in \mathbb{R}^d$ with $|\xi_{q}|/|\xi|\ge\varepsilon_0$, we have
\begin{equation*}
\sigma \left(\tilde{B}\left(x(\xi),|\xi |^{-1}\right)\right)\lesssim \prod_{
i=1, i\ne q}^{d} \min\left\{|\xi|^{-\frac{1}{m_{q,i}\omega_i}},|\xi|^{-\frac{1}{2}}\left(\frac{|\xi_i|}{|\xi|}\right)^{-\frac{m_{q,i}\omega_i-2}
{2(m_{q,i}\omega_i-1)}}\right\},
\end{equation*}
where the implicit constant only depends on $\mathcal{D}$ and $\varepsilon_0$.
\end{lemma}
\begin{proof}
See  Lemma 2.1 in \cite[p. 70]{model2}.
\end{proof}

\begin{lemma}\label{A and B}
Let $1\leq j\le d-1$ be an integer. For the shell $\mathcal D(r,t)$, given any nonzero $n\in\mathbb{Z}^d$, we have
\begin{equation}\label{e2}
\begin{split}
\widehat{\chi}_{\mathcal D(r,t)}(n)\!\!\lesssim\!\!\left\{\begin{array}{ll}
\!\!r^{d-1}|n|^{-1}\sigma\left(\tilde{B}\left(x(n),(r|n|)^{-1}\right)\right)\min\{t|n|,1\}&\!\!\textrm{if $n\!\!\in\mathbb Z^d(j)$},\\
\!\!r^{(d-1)/2}|n|^{-(d+1)/2}(K(x(n)))^{-1/2}\min\{t|n|,1\}&\!\!\textrm{if $n\!\!\in\mathbb Z^d(d)$},
\end{array}
\right.
\end{split}
\end{equation}
where $\mathbb Z^d(j)$ and $\mathbb Z^d(d)$ are defined by \eqref{Zd(j)} and the implicit constants are independent of $r,t$ and $n$. Furthermore, if $n\in \mathbb Z^d(d)$, we have
\begin{equation*}
\widehat{\chi}_{\mathcal D(r,t)}(n)=A(r,t,n)+B(r,t,n)
\end{equation*}
with
\begin{equation}
\begin{split}\label{A(rtn)}
A(r,t,n)=&2\pi^{-1}r^{(d-1)/2}|n|^{-(d+1)/2}(K(x(n)))^{-1/2}  \\
&\cos(2\pi r x(n)\cdot n-\pi(d-1)/4)\sin(\pi tx(n)\cdot n)
\end{split}
\end{equation}
and
\begin{equation}
\begin{split}
B(r,t,n)\lesssim \min\Big\{&r^{(d-1)/2}|n|^{-(d+1)/2}(K(x(n)))^{-1/2}\min\{t|n|,1\}\label{e3},\\
 &r^{d-1-\tilde N}t|n|^{-\tilde N}(K(x(n)))^{-4\tilde N}\Big\}.
\end{split}
\end{equation}
Here $\tilde N$ denotes the integer part of $(d+1)/2$ for simplicity and the implicit constant is independent of $r,t$ and $n$.
\end{lemma}
\begin{proof}
Since
\begin{equation}\label{chiD(r,t)}
\widehat{\chi}_{\mathcal D(r,t)}(n)=(r+t/2)^d\widehat{\chi}_{\mathcal{D}}((r+t/2)n)-(r-t/2)^d\widehat{\chi}_{\mathcal{D}}((r-t/2)n),
 \end{equation}
 combining the divergence theorem, Theorem B in \cite[p. 335]{BNW} and \eqref{svensson} yields
\begin{equation}\label{e0}
\widehat{\chi}_{\mathcal D(r,t)}(n)\lesssim \left\{\begin{array}{ll}
r^{d-1}|n|^{-1}\sigma\left(\tilde{B}\left(x(n),(r|n|)^{-1}\right)\right)&\textrm{if $n\in\mathbb Z^d(j)$}, \\
r^{(d-1)/2}|n|^{-(d+1)/2}(K(x(n)))^{-1/2}&\textrm{if $n\in\mathbb Z^d(d)$},
\end{array}
\right.
\end{equation}
where $1\le j\le d-1$. On the other hand,
\begin{equation}\label{splitD(r,t,n)}
\begin{split}
\widehat{\chi}_{\mathcal D(r,t)}(n)=\left((r+t/2)^d-r^d\right)\widehat{\chi}_{\mathcal{D}}((r+t/2)n)+r^d\big(\widehat{\chi}_{\mathcal{D}}((r+t/2)n)\\
-\widehat{\chi}_{\mathcal{D}}((r-t/2)n)\big)+\left(r^d-(r-t/2)^d\right)\widehat{\chi}_{\mathcal{D}}((r-t/2)n).
\end{split}
\end{equation}
Note that $\nabla\widehat{\chi}_{\mathcal{D}}(n)=-2\pi i\int_{\mathcal{D}}xe^{-2\pi in\cdot x}\,\mathrm dx$, we can handle this integral like $\widehat\chi_{\mathcal D}(n)$ to get
\begin{equation}\label{splitD(r,t,n)2}
\nabla\widehat{\chi}_{\mathcal{D}}(n)\lesssim \left\{\begin{array}{ll}
|n|^{-1}\sigma\left(\tilde{B}(x(n),|n|^{-1})\right)&\textrm{if $n\in\mathbb Z^d(j)$},\\
|n|^{-(d+1)/2}(K(x(n)))^{-1/2}&\textrm{if $n\in\mathbb Z^d(d)$}.
\end{array}
\right.
\end{equation}
Combining \eqref{splitD(r,t,n)} and \eqref{splitD(r,t,n)2} yields
\begin{equation}\label{e1}
\widehat{\chi}_{\mathcal D(r,t)}(n)\lesssim \left\{\begin{array}{ll}
r^{d-1}|n|^{-1}\sigma\left(\tilde{B}\left(x(n),(r|n|)^{-1}\right)\right)t|n|&\textrm{if $n\in\mathbb Z^d(j)$},\\
r^{(d-1)/2}|n|^{-(d+1)/2}(K(x(n)))^{-1/2}t|n|&\textrm{if $n\in\mathbb Z^d(d)$}.
\end{array}
\right.
\end{equation}
Then \eqref{e0} and \eqref{e1} enable us to get \eqref{e2}.

Furthermore, if $n\in \mathbb Z^d(d)$, by using \eqref{expansion of ft} and \eqref{chiD(r,t)} we rewrite
\begin{equation*}
\widehat{\chi}_{\mathcal D(r,t)}(n)=A(r,t,n)+B(r,t,n)
\end{equation*}
with
\begin{equation*}
\begin{split}
A(r,t,n)&=r^{(d-1)/2}|n|^{-(d+1)/2}(a((r+t/2)n)-a((r-t/2)n))
\end{split}
\end{equation*}
and
\begin{equation}\label{B(rtn)}
\begin{split}
B(r,t,n)=&\left((r+t/2)^{(d-1)/2}-r^{(d-1)/2}\right)|n|^{-(d+1)/2}a((r+t/2)n)\\
&+\left(r^{(d-1)/2}-(r-t/2)^{(d-1)/2}\right)|n|^{-(d+1)/2}a((r-t/2)n)\\
&+\left((r+t/2)^{d}-(r-t/2)^{d}\right)E((r+t/2)n)\\
&+(r-t/2)^{d}(E((r+t/2)n)-E((r-t/2)n)).
\end{split}
\end{equation}
It follows from \eqref{a(xi)} that
\begin{equation*}
\begin{split}
A(r,t,n)=&2\pi^{-1}r^{(d-1)/2}|n|^{-(d+1)/2}(K(x(n)))^{-1/2}  \\
&\cos(2\pi r x(n)\cdot n-\pi(d-1)/4)\sin(\pi tx(n)\cdot n).
\end{split}
\end{equation*}
Since $x(n)\cdot n\asymp |n|$, we have
 \begin{equation*}
 A(r,t,n)\lesssim r^{(d-1)/2}|n|^{-(d+1)/2}(K(x(n)))^{-1/2}\min\{t|n|,1\}.
 \end{equation*}
 Combining this inequality with \eqref{e2} gives the first bound in \eqref{e3}.
Let $\tilde N$ be the integer part of $(d+1)/2$. Then by \eqref{est nabla E},
 \begin{equation*}
 (r-t/2)^{d}(E((r+t/2)n)-E((r-t/2)n))\lesssim r^{d-1-\tilde N}t|n|^{-\tilde N}(K(x(n)))^{-4\tilde N}.
 \end{equation*}
Note that the first three terms on the right side of \eqref{B(rtn)} are also bounded by $O\left(r^{d-1-\tilde N}t|n|^{-\tilde N}(K(x(n)))^{-4\tilde N}\right)$, so we have the second bound in \eqref{e3}. This finishes the proof.
\end{proof}

\begin{lemma}\label{two}
 Given any $\epsilon>0$, assume that $\psi(\xi)$ is a smooth function supported in $\cap_{i=1}^d\{\xi\in\mathbb R^d:\epsilon\le |\xi|\le 1/\epsilon, |\xi_i|/|\xi|\ge \epsilon\}$ and satisfies that for any multi-index $\nu$ with $|\nu|=j$, there is a $\gamma(j)\in \mathbb{N}$ and a constant $C_j$ such that
\begin{equation*}
\begin{split}
\left|\frac{\partial^{|\nu|}\psi(\xi)}{\partial \xi^{\nu}}\right|\le C_j\epsilon^{-\gamma(j)}.
\end{split}
\end{equation*}
Then for the domain $\mathcal D$, given any $k\in\mathbb{N},\zeta\in \mathbb R^d$ and $\lambda>0$, there is a $\mathbf{\gamma}_k\in\mathbb{N}$ such that
 \begin{equation}\label{oscillatory int}
 \begin{split}
&\int_{\mathbb R^d}\psi(\xi)e^{2\pi i \lambda(2x(\xi)\cdot \xi-\zeta\cdot \xi)}\,\mathrm d\xi\\
&\lesssim \epsilon^{-\gamma_k}\min\left\{(\lambda|\zeta|)^{-(d-1)/2},(\lambda\,\mathrm{dist}\{\zeta,\partial(2\mathcal D)\})^{-k}\right\},
\end{split}
\end{equation}
where the implicit constant is independent of $\lambda$ and $\zeta$.
\end{lemma}
\begin{proof}
Since
\begin{equation*}
|\nabla(2x(\xi)\cdot \xi-\zeta\cdot \xi)|=|2x(\xi)-\zeta|\ge\mathrm{dist\{\zeta,2\partial\mathcal D\}},
\end{equation*}
by integration by parts there is an integer $\gamma_1(k)>0$ such that
\begin{equation}\label{two1}
\left|\int_{\mathbb R^d}\psi(\xi)e^{2\pi i \lambda(2x(\xi)\cdot \xi-\zeta\cdot \xi)}\,\mathrm d\xi\right|\lesssim \epsilon^{-\gamma_1(k)}(\lambda\,\mathrm{dist}\{\zeta,2\partial\mathcal D\})^{-k}.
\end{equation}

On the other hand,
\begin{equation*}
\begin{split}
\int_{\mathbb R^d}\psi(\xi)e^{2\pi i \lambda(2x(\xi)\cdot \xi-\zeta\cdot \xi)}\,\mathrm d\xi
\!&=\!\!\int_0^{+\infty}\!\!e^{4\pi i \lambda h}\!\!\int_{x(\xi)\cdot \xi=h}\!\frac{\psi(\xi)e^{-2\pi i \lambda\zeta\cdot \xi}}{|x(\xi)|}\,\mathrm {dH}_{d-1}(\xi)\,\mathrm dh\\
&\lesssim\int_{\epsilon}^{1/\epsilon}\left|\int_{x(\xi)\cdot \xi=h}\frac{\psi(\xi)e^{-2\pi i \lambda\zeta\cdot \xi}}{|x(\xi)|}
\,\mathrm {dH}_{d-1}(\xi)\right|\,\mathrm dh,
\end{split}
\end{equation*}
where $\text H_{d-1}$ is the $(d-1)$-dimensional Hausdorff measure and Theorem 3.13 in \cite[p. 140]{coarea} gives the first equality above since $f(\xi):=x(\xi)\cdot \xi$ is Lipschitz continuous and the Jacobian $|Df(\xi)|=|x(\xi)|>0$. The support of $\psi$ implies $\epsilon\lesssim h\lesssim 1/\epsilon $ and $K(x(\xi))\gtrsim \epsilon^{d-1}$ by Lemma \ref{relation}. Therefore given any multi-index $\nu$ with $|\nu|=j$, by Lemma \ref{lemma1} there is an integer $\gamma_2(j)>0$ such that
\begin{equation*}
\begin{split}
\frac{\partial^{|\nu|}\left(\psi(\xi)|x(\xi)|^{-1}\right)}{\partial \xi^{\nu}}
\lesssim \epsilon^{-\gamma_2(j)}.
\end{split}
\end{equation*}
Let $\widetilde K(\xi)$ denote the Gaussian curvature of hypersurface $x(\xi)\cdot \xi=h$ at the point $\xi$. If $K(x(\xi))\ne 0$, we have
\begin{equation*}
\widetilde K(\xi)\asymp h^{-(d-1)}(K(x(\xi)))^{-1}\footnote{ We put the calculation of $\widetilde K(\xi)$ in Appendix \text{3}.}.
 \end{equation*}

If there is a $\xi_0\in \mathrm{supp}\psi\cap\{\xi\in \mathbb R^d\setminus \{\vec 0\}:x(\xi)\cdot\xi=h\}$ such that the outer normal of the hypersurface $x(\xi)\cdot \xi=h$ at $\xi_0$ is parallel to $\pm\zeta$, then there is an integer $\tilde \gamma_3>0$ such that
\begin{equation}\label{two2}
\begin{split}
\int_{x(\xi)\cdot \xi=h}\frac{\psi(\xi)e^{-2\pi i \lambda\zeta\cdot \xi}}{|x(\xi)|}\,\mathrm {dH}_{d-1}(\xi)
&\lesssim \epsilon^{-\tilde \gamma_3} \left(\widetilde K(\xi_0)\right)^{-\frac{1}{2}}(\lambda|\zeta|)^{-\frac{d-1}{2}}\\
&\lesssim \epsilon^{-\tilde \gamma_3} h^{(d-1)/2}(\lambda|\zeta|)^{-(d-1)/2},
\end{split}
\end{equation}
where the implicit constant is independent of $\lambda$ and $\zeta$.
If there is no such a point $\xi_0$, applying integration by parts gives that for any $k\in \mathbb N$, there is an integer  $\gamma_4(k)>0$ such that
\begin{equation}\label{two3}
\int_{x(\xi)\cdot \xi=h}\psi(\xi)|x(\xi)|^{-1}e^{-2\pi i \lambda\zeta\cdot \xi}\,\mathrm {dH}_{n-1}(\xi)\lesssim \epsilon^{-\gamma_4(k)}(\lambda|\zeta|)^{-k},
\end{equation}
where the implicit constant is independent of $\lambda$ and $\zeta$. Thus \eqref{oscillatory int} follows from \eqref{two1}, \eqref{two2} and \eqref{two3}.
\end{proof}



\section{Proof of theorem \ref{variance}}
As we said in the introduction that this paper extends the result in \cite{CGG}, the structure and many ideas in this part follow that paper.

By Parseval's identity or Lemma 2.1 in \cite{CGG} we have
\begin{equation}
\begin{split}\label{variance2}
\int_{\mathbb{T}^d}\left|N_{\mathcal D(r,t)}(u)-\vol(\mathcal D(r,t))\right|^2\,\mathrm du
&=\sum _{n\in\mathbb{Z}^d\setminus\{0\}}|\widehat{\chi}_{\mathcal D(r,t)}(n)|^2\\
&=\sum _{j=1}^{d}\sum _{n\in\mathbb{Z}^d(j)}|\widehat{\chi}_{\mathcal D(r,t)}(n)|^2,
\end{split}
\end{equation}
where $\mathbb{Z}^d(j)$ is defined by \eqref{Zd(j)}.

\begin{lemma} \label{lemmaZd(j)}
If $1< r<\infty$, $0<t\le Cr^{-\alpha}$ for some constant $C>0$ and $\alpha>\max_{1\le i\le d-1}\{\alpha_i\}$ with $\alpha_i$ defined by \eqref{alphaj}, then given any $1\le j\le d-1$, there is a constant $\beta_j>0$ such that
\begin{equation}\label{summation over Zd(j)}
\sum _{n\in\mathbb{Z}^d(j)}|\widehat{\chi}_{\mathcal D(r,t)}(n)|^2\lesssim r^{d-1}t^{1+\beta_j},
\end{equation}
where the implicit constant is independent of $r$ and $t$.
\end{lemma}
\begin{proof}
Given any $1\le j\le d-1$, let $C_{j,1}=\sum_{l=j+1}^{d}2/(m_{1,l}\omega_l)$ and
\begin{equation*}
S_{j,1}=\{(n_1,\ldots,n_j,0,\ldots,0):n_1,\ldots,n_j\in\mathbb N, n_1/|n|\ge \varepsilon_0\}.
\end{equation*}
Then $S_{j,1}$ is a subset of $\mathbb{Z}^d(j)$. We only prove \eqref{summation over Zd(j)} when the summation is over all $n\in S_{j,1}$ as all other cases can be handled similarly. From
\eqref{e2} and Lemma \ref{lemmacap}, we have
\begin{equation*}
\begin{split}
&\sum _{n\in S_{j,1}}|\widehat{\chi}_{\mathcal D(r,t)}(n)|^2\lesssim\\
&\left\{\begin{array}{ll}
\!\!r^{2d-2-C_{1,1}}\sum\limits_{n_1=1}^{+\infty}n_1^{-2-C_{1,1}}\min\{t^2n_1^2,1\} & \text{if $j=1$},\\
\!\!r^{2d-1-j-C_{j,1}}\!\!\sum\limits_{n\in S_{j,1}}\!\!|n|^{-1-j-C_{j,1}}\!\!\prod\limits_{l=2}^{j}\!\left(\frac{n_l}{|n|}\right)^
{-\frac{m_{1,l}\omega_l-2}{m_{1,l}\omega_l-1}}\min\{t^2|n|^2,1\} & \text{if $j\ge2$}.
\end{array}
\right.
\end{split}
\end{equation*}

If $j=1$, note that
\begin{equation}\label{radial integral}
\begin{split}
\sum_{n_1=1}^{+\infty}n_1^{-2-C_{1,1}}\min\{t^2n_1^2,1\}&\lesssim\int_{1}^{+\infty}s^{-2-C_{1,1}}\min\{t^2s^2,1\}\,\mathrm ds\\
&\lesssim \left\{\begin{array}{ll}
t^{1+C_{1,1}}&\text{if $C_{1,1}<1$},\\
t^2\log\left(t^{-1}\right)&\text{if $C_{1,1}=1$},\\
t^2&\text{if $C_{1,1}>1$}.
\end{array}
\right.
\end{split}
\end{equation}
Then if $\alpha>\alpha_1$ and $t\lesssim r^{-\alpha}$, there is a constant $\beta_1>0$ such that
\begin{equation*}
\sum _{n\in S_{1,1}}|\widehat{\chi}_{\mathcal D(r,t)}(n)|^2\lesssim r^{d-1}t^{1+\beta_1}.
\end{equation*}

If $2\le j\le d-1$, comparing the summation to integral yields
\begin{equation*}
\sum _{n\in S_{j,1}}|\widehat{\chi}_{\mathcal D(r,t)}(n)|^2\lesssim r^{2d-1-j-C_{j,1}}\prod _{i=0}^{j-1}I_i,
\end{equation*}
where
\begin{align*}
I_0=\int_1^{+\infty}s^{-2-C_{j,1}}\min\{t^2s^2,1\}\,\mathrm ds
\end{align*}
has the similar bounds in \eqref{radial integral},
\begin{align}\label{first angle}
I_1=\int_{0}^{\pi/2}(\sin\phi_1)^{-\sum_{l=2}^{j}\frac{m_{1,l}\omega_l-2}{m_{1,l}\omega_l-1}+j-2}
\,\mathrm d\phi_1 \lesssim 1
\end{align}
and
\begin{align}\label{ith angle}
I_i=\int_{0}^{\pi/2}(\sin\phi_i)^{-\sum_{l=i+1}^{j}\frac{m_{1,l}\omega_l-2}
{m_{1,l}\omega_l-1}+j-i-1} (\cos\phi_i)^{-\frac{m_{1,i}\omega_i-2}{m_{1,i}\omega_i-1}}\,\mathrm d\phi_i\lesssim 1
\end{align}
for $2\le i\le j-1$. So if $\alpha>\alpha_j$ and $t\lesssim r^{-\alpha}$, we have
\begin{equation*}
\sum_{n\in S_{j,1}}|\widehat{\chi}_{\mathcal D(r,t)}(n)|^2\lesssim r^{d-1}t^{1+\beta_j}
\end{equation*}
 for some constant $\beta_j>0$ . This finishes the proof.
\end{proof}

We will use the following Lemma \ref{l4}--\ref{estimateY(r,t,n)} to obtain
\begin{equation}\label{Zd(d)}
\sum _{n\in\mathbb{Z}^d(d)}|\widehat{\chi}_{\mathcal D(r,t)}(n)|^2= \vol(\mathcal D)dr^{d-1} t+O\left(r^{d-1}t^{1+\beta}\right)
\end{equation}
for some constant $\beta>0$. Then Theorem \ref{variance} follows from \eqref{variance2}, \eqref{Zd(d)} and Lemma \ref{lemmaZd(j)}.
\begin{lemma}\label{l4}
If $1< r<\infty$, $0<t\le Cr^{-\alpha}$ for some constant $C>0$ and $\alpha>\max_{1\le i\le d-1}\{\alpha_i\}$ with $\alpha_i$ defined by \eqref{alphaj}, then  there is a constant $\delta_1>0$ such that
\begin{equation*}
\begin{split}
\sum _{n\in\mathbb{Z}^d(d)}|\widehat{\chi}_{\mathcal D(r,t)}(n)|^2
=X(r,t)+Y(r,t)+Z(r,t),
\end{split}
\end{equation*}
where
\begin{align}\label{X(r,t)}
X(r,t)=2\pi^{-2}r^{d-1}\sum _{n\in \mathbb Z^d(d)}|n|^{-d-1}(K(x(n)))^{-1}\sin^2(\pi tx(n)\cdot n),
\end{align}
\begin{equation}\label{Y(r,t)}
\begin{split}
Y(r,t)=2\pi^{-2}r^{d-1}\sum _{n\in \mathbb Z^d(d)}&|n|^{-d-1}(K(x(n)))^{-1}\sin^2(\pi tx(n)\cdot n)\\
&\cos(4\pi rx(n)\cdot n-(d-1)\pi/2)
\end{split}
\end{equation}
and
\begin{equation*}
Z(r,t)\lesssim r^{d-1}t\left(r^{-1}t\right)^{\delta_1}.
\end{equation*}
Here the implicit constant is independent of $r$ and $t$.
\end{lemma}
\begin{proof}
By using Lemma  \ref{A and B}, we have
\begin{align*}
\sum _{n\in\mathbb{Z}^d(d)}|\widehat{\chi}_{\mathcal D(r,t)}(n)|^2=\sum _{n\in \mathbb{Z}^d(d)}&\Big(|A(r,t,n)|^2+|B(r,t,n)|^2\\
&+\overline{A(r,t,n)}B(r,t,n)+A(r,t,n)\overline{B(r,t,n)}\Big).
\end{align*}
Combining \eqref{A(rtn)} and the power-reduction formula for the cosine yields
\begin{align*}
\sum_{n\in \mathrm Z^d(d)}|A(r,t,n)|^2= X(r,t)+Y(r,t).
\end{align*}
Then to estimate $Z(r,t)$ it suffices to show that
 \begin{equation}\label{Z(r,t)}
\sum _{n\in \mathbb{Z}^d(d)}\left(|B(r,t,n)|^2+\left|{A(r,t,n)}B(r,t,n)\right|\right)\lesssim r^{d-1}t\left(r^{-1}t\right)^{\delta_1}.
\end{equation}

Let
\begin{equation*}
S_{d,1}=\{(n_1,\ldots,n_d):n_1,\ldots,n_d\in\mathbb N, n_1/|n|\ge \varepsilon_0\}.
\end{equation*}
Then $S_{d,1}$ is a subset of $\mathbb{Z}^d(d)$. We only prove \eqref{Z(r,t)} when the summation is over all $n\in S_{d,1}$ as all other cases can be handled similarly. From Lemma \ref{A and B}, we have
\begin{equation}\label{bound of |B|^2}
\begin{split}
\sum _{n\in S_{d,1}}|B(r,t,n)|^2\lesssim I+II
,
\end{split}
\end{equation}
where
\begin{align*}
I&=r^{d-1}\sum _{n\in S_{d,1},K(x(n))<\epsilon}|n|^{-(d+1)}(K(x(n)))^{-1}\min\{t^2|n|^2,1\},\\
II&=r^{-\tilde N+3(d-1)/2}\sum _{n\in S_{d,1},K(x(n))\ge\epsilon}\begin{aligned}[t]&|n|^{-\tilde N-(d+1)/2}t(K(x(n)))^{-4\tilde N-1/2}\\
&\min\{t|n|,1\}
\end{aligned}
\end{align*}
and $\epsilon >0$ is determined below.
 It is easy to get that
\begin{equation}
\begin{split}\label{bound of b2}
II&\lesssim r^{-\tilde N+3(d-1)/2}\int_{1}^{+\infty}s^{-\tilde N+(d-3)/2}t\min\{ts,1\}\,\mathrm ds\int_{S^{d-1}}\epsilon^{-4\tilde N-1/2}\,\mathrm dS\\
&\lesssim r^{d-3/2}t^{3/2}\epsilon^{-4\tilde N-1/2},
\end{split}
\end{equation}
where $\mathrm dS$ is the induced measure on the sphere.

For $I$, by Lemma \ref{relation} we have
\begin{equation*}
\prod _{i=2}^{d}\left(n_i/|n|\right)^{\frac{m_{1,i}\omega_i-2}{m_{1,i}\omega_i-1}}\lesssim \epsilon,
\end{equation*}
which implies that there exists an integer $2\le i_0\le d$ such that
\begin{equation}\label{ni0}
(n_{i_0}/|n|)^{\frac{m_{1,i_0}\omega_{i_0}-2}{m_{1,i_0}\omega_{i_0}-1}}\lesssim \epsilon^{1/(d-1)}.
 \end{equation}
Let
\begin{equation*}
\omega=\max_{1\le q,l\le d}\{m_{q,l}\omega_l\}
\end{equation*}
and $c'$ be a constant satisfying $0<c'<1/(\omega-1)$. Then
\begin{equation*}
I\lesssim r^{d-1}t\int_{S^{d-1}}|\nu_{i_0}|^{-\frac{m_{1,i_0}\omega_{i_0}-2}{m_{i_0,l}\omega_{i_0}-1}-c'}\epsilon^{c_{i_0}}\prod _{l=2,l\ne i_0}^{d-1}|\nu_l|^{-\frac{m_{1,l}\omega_l-2}{m_{1,l}\omega_l-1}}\,\mathrm dS(\nu),
\end{equation*}
where $c_{i_0}=\frac{c'(m_{1,i_0}\omega_{i_0}-1)}{(d-1)(m_{1,i_0}\omega_{i_0}-2)}$. Similar to \eqref{first angle} and \eqref{ith angle}, we have
\begin{align} \label{bound of B2}
I\lesssim r^{d-1}t\epsilon^{c_{i_0}}.
\end{align}
Hence if we take $\epsilon=\left(r^{-1}t\right)^{1/(2c_{i_0}+8\tilde N+1)}$, combining \eqref{bound of |B|^2}, \eqref{bound of b2} and \eqref{bound of B2} yields
\begin{align}\label{FBO |B|^2}
\sum _{n\in S_{d,1}}|B(r,t,n)|^2\lesssim
 r^{d-1}t\left(r^{-1}t\right)^{c_{i_0}/(2c_{i_0}+8\tilde N+1)},
\end{align}
where the implicit constant is independent of $r$ and $t$.

On the other hand, we can estimate $\sum_{n\in S_{d,1}}|A(r,t,n)B(r,t,n)|$ in the same way to get the bound in \eqref{FBO |B|^2}.
 This finishes the proof.
\end{proof}

\begin{lemma}
If $1< r<\infty$, $0<t\le Cr^{-\alpha}$ for some constant $C>0$ and $\alpha>\max_{1\le i\le d-1}\{\alpha_i\}$ with $\alpha_i$ defined by \eqref{alphaj}, then  there is a constant $\delta_2>0$ such that
\begin{equation*}
X(r,t)=\vol(\mathcal D)dr^{d-1}t+O\left(r^{d-1}t^{1+c'}
+r^{d-1}t\left(t\log\left(t^{-1}\right)\right)^{\delta_2}\right),
\end{equation*}
where $c'$ is a constant satisfying $0<c'<1/(\omega-1)$ and the implicit constant is independent of $r$ and $t$.
\end{lemma}
\begin{proof}
Let $T$ denote the union of coordinate planes. By \eqref{X(r,t)}
\begin{align}\label{X(r,t)2}
X(r,t)=I_0-2\pi^{-2}r^{d-1}\sum _{i=1}^4I_i,
\end{align}
where
\begin{equation*}
I_0=2\pi^{-2}r^{d-1}\int_{\mathbb R^d\backslash T}|\xi|^{-d-1}(K(x(\xi)))^{-1}\sin^2(\pi tx(\xi)\cdot \xi)\,\mathrm d\xi,
\end{equation*}
\begin{equation*}
I_1=\sum _{j=1}^{d-1}\sum _{n\in \mathbb Z^d(j)}\int_{(\mathbb T^d+n)\backslash T}|\xi|^{-d-1}(K(x(\xi)))^{-1}\sin^2(\pi tx(\xi)\cdot \xi)\,\mathrm d\xi,
\end{equation*}
\begin{equation*}
\begin{split}
I_2=\sum _{n\in \mathbb Z^d(d)}\int_{\mathbb T^d}&|\xi+n|^{-d-1}\left((K(x(\xi+n)))^{-1}-(K(x(n)))^{-1}\right)\\
&\sin^2(\pi tx(\xi+n)\cdot (\xi+n))\,\mathrm d\xi,
\end{split}
\end{equation*}
\begin{equation*}
\begin{split}
I_3=\sum _{n\in \mathbb Z^d(d)}(K(x(n)))^{-1}\int_{\mathbb T^d}&|\xi+n|^{-d-1}\big(\sin^2(\pi tx(\xi+n)\cdot(\xi+n))\\
&-\sin^2(\pi tx(n)\cdot n)\big)\,\mathrm d\xi
\end{split}
\end{equation*}
and
\begin{equation*}
I_4=\sum _{n\in \mathbb Z^d(d)}(K(x(n)))^{-1}\sin^2(\pi tx(n)\cdot n)\int_{\mathbb T^d}|\xi+n|^{-d-1}-|n|^{-d-1}\,\mathrm d\xi.
\end{equation*}

Firstly,
\begin{equation}
\begin{split}\label{quantity}
&I_0=2\pi^{-1}r^{d-1} t\int_{0}^{+\infty}(\sin^2s)/s^{2}\,\mathrm ds\int_{S^{d-1}\backslash T}(K(x(v)))^{-1}x(v)\cdot v\,\mathrm dS(v)\\
&\quad=r^{d-1} t\int_{\partial \mathcal D\backslash T}x\cdot n(x)\,\mathrm d \sigma(x)\\
&\quad=\vol(\mathcal D)dr^{d-1} t,
\end{split}
\end{equation}
where in the last two equalities we use the facts that if $K(x(\nu))\ne 0$, there is a small neighborhood of $\nu$ such that $(K(x(\nu)))^{-1}\mathrm dS(\nu)=\mathrm d\sigma(x)$ and $x(v)\cdot v=x\cdot n(x)$ is the height of the cone with vertex the origin and base $\mathrm d\sigma$.

We only estimate $I_1$ when the summation is over all $n\in S_{j,1}$ with $1\le j\le d-1$ and other cases can be handled similarly.
Given any $n\in  S_{j,1}$ and $\xi\in (n+\mathbb T^d)\backslash T$, we have $0<|\xi_d|\le 1/2$ and
\begin{equation*}
\begin{split}
&\sum _{n\in S_{j,1}}\int_{(\mathbb T^d+n)\backslash T}|\xi|^{-d-1}(K(x(\xi)))^{-1}\sin^2(\pi tx(\xi)\cdot \xi)\,\mathrm d\xi\\
&\lesssim\!\int_{1}^{+\infty}\!s^{-2-c'}\min\{t^2s^2,1\}\,\mathrm ds\!\int_{S^{d-1}}\!|\nu_d|^{-\frac{m_{1,d}\omega_{d}-2}{m_{1,d}\omega_d-1}-c'}\prod_{l=2}^{d-1} |\nu_l|^{-\frac{m_{1,l}\omega_{l}-2}{m_{1,l}\omega_l-1}}\,\mathrm dS(\nu)\\
&\lesssim  t^{1+c'}.
\end{split}
\end{equation*}

If $\xi\in \mathbb T^d$ and $n\in \mathbb Z^d(d)$, then $|n_l|\asymp|n_l+\xi_l|$ for $1\le l\le d$. We only estimate $I_2$ when the summation is over all $n\in S_{d,1}$ and other cases can be handled similarly. Given any $n\in S_{d,1}$, by Lemma \ref{lemma1} and lemma \ref{relation} we have
\begin{equation*}
\left|(K(x(n+\xi)))^{-1}\!-(K(x(n)))^{-1}\right|\lesssim\min\left\{(K_{L}(n))^{-1},|n|^{-1}(K_{L}(n))^{-3}\right\},
\end{equation*}
where
$
K_{L}(n)=\prod _{l=2}^d(n_l/|n|)^{\frac{m_{1,l}\omega_l-2}{m_{1,l}\omega_l-1}}
$
and using the subscript ``L'' is because that $K_{L}(n)$ is a ``lower'' bound of $K(x(n))$ in \eqref{relation between point and direction} if $n\in S_{d,1}$.
Hence
\begin{align*}
&\sum _{n\in S_{d,1}}\int_{\mathbb T^d}\begin{aligned}[t]&|\xi+n|^{-d-1}\left((K(x(\xi+n)))^{-1}-(K(x(n)))^{-1}\right)\\
&\sin^2(\pi tx(\xi+n)\cdot (\xi+n))\mathrm d\xi
\end{aligned}\\
&\lesssim \sum _{n\in S_{d,1},K_{L}(n)<\epsilon}|n|^{-d-1}
(K_{L}(n))^{-1}\min\{t^2|n|^2,1\}\\
&\quad+\sum _{n\in S_{d,1},K_{L}(n)\ge\epsilon}|n|^{-d-2}(K_{L}(n))^{-3}\min\{t^2|n|^2,1\}.
\end{align*}
Similar to estimating the upper bound of $\sum_{n\in S_{d,1}}$ $|B(r,t,n)|^2$ in Lemma \ref{l4}, there is a $0<\delta'_2<1$ such that the summations above are bounded by $ O\left(t(t\log(1/t))^{\delta'_2}\right)$.

For $I_3$, given any $s_1,s_2\in \mathbb R$ and nonzero $\xi_1,\xi_2\in \mathbb R^d$, a simple computation shows that
\begin{equation*}
\sin^2(s_1)-\sin^2(s_2)=\sin(s_1+s_2)\sin(s_1-s_2)
\end{equation*}
and
\begin{equation*}
|x(\xi_1)\cdot\xi_1-x(\xi_2)\cdot(\xi_2)|\lesssim |\xi_1-\xi_2|,
\end{equation*}
where the implicit constant is independent of $\xi_1$ and $\xi_2$. So
\begin{equation*}
I_3\lesssim \sum _{n\in \mathbb Z^d(d)}|n|^{-d-1}t(K(x(n)))^{-1}\min\{t|n|,1\}\lesssim t^2\log(1/t).
\end{equation*}

It is easy to check that
\begin{align*}
I_4\lesssim\sum _{n\in \mathbb Z^d(d)}|n|^{-d-2}(K(x(n)))^{-1}\min\{t^2|n|^2,1\}\lesssim  t^2\log(1/t).
\end{align*}
Then combining \eqref{X(r,t)2}, \eqref{quantity} and all bounds for $I_i$ with $1\le i\le 4$ gives the desired result.
\end{proof}

In the next part, we use some smooth functions to split $Y(r,t)$ and estimate each part separately. Let $\varphi(s)$ be a smooth cut-off function with $\supp\varphi=[-3/4,3/4]$ and $\varphi(s)=1$ if $|s|\le 1/4$. Then given any $0<\epsilon<3^{-1/2}$, define
\begin{align*}
\phi_1^{\epsilon}(x)=\prod_{i=1}^{d}\left(1-\varphi\left(\frac{x_i}{\epsilon|x|}\right)\right)
\end{align*}
and
 \begin{equation*}
\phi_2^{\epsilon}(s)=\left\{\begin{array}{ll}
0&\textrm{if $s\le 0$},\\
\varphi(\epsilon s)-\varphi(s/\epsilon)&\textrm{if $s>0$}.
\end{array}
\right.
\end{equation*}
Then we have
\begin{gather*}
 \supp\phi_1^{\epsilon}\subseteq\bigcap _{i=1}^{d}\left\{x\in\mathbb R^d\backslash\{\vec 0\}:|x_i|/|x|\ge\epsilon/4\right\},\\
\phi_1^{\epsilon}(x)=1\quad \text{if}\quad
x\in\bigcap _{i=1}^d\left\{x\in\mathbb R^d\backslash\{\vec 0\}:|x_i|/|x|\ge3\epsilon/4\right\},\\
\supp\phi_2^{\epsilon}=\{s\in\mathbb R: \epsilon/4\le s\le3/(4\epsilon)\},
\end{gather*}
 and $\phi_2^{\epsilon}(s)=1$ if
$3\epsilon/4\le s\le 1/(4\epsilon)$.

\begin{lemma}\label{estimateY(r,t,n)}
If $1< r<\infty$, $0<t\le Cr^{-\alpha}$ for some constant $C>0$ and $\alpha>\max_{1\le i\le d-1}\{\alpha_i\}$ with $\alpha_i$ defined by \eqref{alphaj}, then  there exists a constant $0<\delta_3<1$ such that
\begin{equation*}
Y(r,t)\lesssim r^{d-1}t^{1+\delta_3},
\end{equation*}
where the implicit constant is independent of $r$ and $t$.
\end{lemma}
\begin{proof}
By \eqref{Y(r,t)} we rewrite $Y(r,t)=I+II$ with
\begin{equation*}
\begin{split}
I=2\pi^{-2}r^{d-1}\sum _{n\in \mathbb Z^d(d)} &\phi_1^{\epsilon}(n)|n|^{-d-1}(K(x(n)))^{-1}\sin^2(\pi tx(n)\cdot n)\\
&\cos\left(4\pi rx(n)\cdot n-(d-1)\pi/2\right)
\end{split}
\end{equation*}
and
\begin{equation*}
\begin{split}
II=2\pi^{-2}r^{d-1}\sum _{n\in \mathbb Z^d(d)} &\left(1-\phi_1^{\epsilon}(n)\right)|n|^{-d-1}(K(x(n)))^{-1}\sin^2(\pi tx(n)\cdot n)\\
&\cos\left(4\pi rx(n)\cdot n-(d-1)\pi/2\right),
\end{split}
\end{equation*}
where $\epsilon>0$ is determined below.

If $n\in \mathbb Z^d(d)\cap \mathrm{supp}\left(1-\phi_1^{\epsilon}(n)\right)$, we have
$|n_j|/|n|\le3\epsilon/4$
 for some $1\le j\le d$. Similar to the case of estimating $\sum_{n\in S_{d,1}}|B(r,t,n)|^2 $ under the condition \eqref{ni0} in Lemma \ref{l4}, there is a $0<\Lambda<1$ such that
\begin{equation}\label{II}
II\lesssim r^{d-1}t\epsilon^{\Lambda}.
 \end{equation}

To study $I$, we rewrite $I=2\pi^{-2}(I_1+I_2)$ with
\begin{align*}
I_1=r^{d-1}\sum _{n\in \mathbb Z^d(d)}&\phi_2^{\epsilon}(t|n|)\phi_1^{\epsilon}(n)|n|^{-d-1}(K(x(n)))^{-1}\sin^2(\pi tx(n)\cdot n)\\
&\cos\left(4\pi rx(n)\cdot n-(d-1)\pi/2\right)
\end{align*}
and
\begin{align*}
I_2=r^{d-1}\sum _{n\in \mathbb Z^d(d)}&\left(1-\phi_2^{\epsilon}(t|n|)\right)\phi_1^{\epsilon}(n)|n|^{-d-1}(K(x(n)))^{-1}\sin^2(\pi tx(n)\cdot n)\\
&\cos\left(4\pi rx(n)\cdot n-(d-1)\pi/2\right).
\end{align*}
Then we have
\begin{equation}
\begin{split}\label{I2}
I_2&\lesssim \begin{aligned}[t]& r^{d-1}\sum _{n\in\mathbb Z^d(d),|n|\le 3\epsilon/(4t)}|n|^{-d+1}(K(x(n)))^{-1}t^2\\
&+r^{d-1}\sum _{n\in\mathbb Z^d(d),|n|\ge 1/(4\epsilon t)}|n|^{-d-1}(K(x(n)))^{-1}\end{aligned}\\
&\lesssim r^{d-1}t\epsilon.
\end{split}
\end{equation}

By Poisson summation formula
\begin{align*}
 I_1&= r^{d-1}t\sum _{n\in\mathbb Z^d}\begin{aligned}[t]&t^d\phi_2^{\epsilon}(t|n|)\phi_1^{\epsilon}(tn)(t|n|)^{-d-1}(K(x(tn)))^{-1}\sin^2(\pi x(tn)\cdot tn)\\
&\cos\left(4\pi rt^{-1}x(tn)\cdot tn-(d-1)\pi/2\right)\end{aligned}\\
&=r^{d-1}t\sum _{n\in\mathbb Z^d}\hat f(n/t),
\end{align*}
with
\begin{align*}
f(\xi)=&\phi_2^{\epsilon}(|\xi|)\phi_1^{\epsilon}(\xi)|\xi|^{-d-1}(K(x(\xi)))^{-1}\sin^2(\pi x(\xi)\cdot \xi)\\
&\cos\left(4\pi r t^{-1}x(\xi)\cdot \xi-(d-1)\pi/2\right).
\end{align*}
We only need to estimate the summation $\sum _{n\in\mathbb Z^d}\hat f(n/t)$. Let
\begin{equation*}
\phi(\xi)=\phi_2^{\epsilon}(|\xi|)\phi_1^{\epsilon}(\xi)|\xi|^{-d-1}(K(x(\xi)))^{-1}\sin^2(\pi x(\xi)\cdot \xi).
\end{equation*}
Then
\begin{equation}\label{Y1}
\begin{split}
\hat f(n/t)
=&2^{-1}e^{-(d-1)\pi i/2}\int_{\mathbb R^d}\phi(\xi)e^{2\pi irt^{-1}(2x(\xi)\cdot \xi-\xi\cdot n/r)}\,\mathrm d\xi\\
&+2^{-1}e^{(d-1)\pi i/2}\int_{\mathbb R^d}\phi(\xi)e^{-2\pi irt^{-1}(2x(\xi)\cdot \xi+\xi\cdot n/r)}\,\mathrm d\xi.
\end{split}
\end{equation}
Given any $\xi\in \supp\phi$, we have $\epsilon\lesssim|\xi|\lesssim 1/\epsilon$ and $(K(x(\xi)))^{-1}\lesssim \epsilon^{-(d-1)}$ by Lemma \ref{relation}. Hence by Lemma \ref{lemma1}, given any multi-index $\nu$, we have
\begin{equation*}
\begin{split}
\frac{\partial^{|\nu|}\phi(\xi)}{\partial \xi^{\nu}}
\lesssim\epsilon^{-d-1-(2d-1)|\nu|}.
\end{split}
\end{equation*}
 Applying Lemma \ref{two} with $\lambda=rt^{-1}$ and $\zeta=\pm r^{-1}n$ to \eqref{Y1} yieids
\begin{equation*}
\hat f(n/t)\lesssim
\epsilon^{-\tilde\gamma_k}\min\left\{(|n|t^{-1})^{-(d-1)/2},
(t^{-1}\,\mathrm{dist}\{n,\partial(2r\mathcal D)\})^{-k}\right\},
\end{equation*}
where $k,\tilde\gamma_k\in \mathbb N$ and the implicit constant is independent of $r,t$ and $n$. Here we choose $k>\max\left\{d,\left(\alpha-\frac{d-1}{\omega'}\right)^{-1}(d-1)\left(\frac{1+\alpha}{2}+\frac{1}{\omega'}\right)\right\}$.
Given any $0\le a<b$, denote
 \begin{equation*}
\{\mathrm a\le d(n)<b\}=\{n\in \mathbb Z^d:a\le \dist\{n,\partial{(2r\mathcal D)}\}<b\}
\end{equation*}
and $s=r^{-(d-1)/\omega'}$ with $\omega'=\max_{1\le q,l\le d}\{m_{q,l}\omega_l,d+1\}$. Then

\begin{align*}
\sum _{n\in\mathbb Z^d}\hat f(n/t)\lesssim&\sum_{\{0\le\mathrm d(n)<s\}}\!\!\epsilon^{-\tilde\gamma_k}\left(|n|t^{-1}\right)^{-(d-1)/2}+\sum_{\{s\le \mathrm d(n)< 1\}}\epsilon^{-\tilde\gamma_k}\left(st^{-1}\right)^{-k}\\
&+\sum _{l=1}^{+\infty}\sum _{\left\{2^{l-1}\le\mathrm d(n)< 2^l\right\}}\epsilon^{-\tilde\gamma_k}\left(2^{l-1}t^{-1}\right)^{-k}.
\end{align*}
 Note that Lemma 29 in \cite[p. 195]{convexdomain} implies that if $n\in\{0\le \mathrm d(n)< s\}$, then $|n|\ge r$.
Hence
\begin{equation}\label{ft1}
\begin{split}
\sum _{n\in\mathbb Z^d}\hat f(n/t)\lesssim&\epsilon^{-\tilde\gamma_k}(r^{-1}t)^{(d-1)/2}\sum_{\{0\le\mathrm d(n)<s\}}1+\epsilon^{-\tilde\gamma_k}s^{-k}t^k\sum_{\{s\le\mathrm d(n)< 1\}}1\\
&+ \epsilon^{-\tilde\gamma_k}t^k\sum _{l=1}^{+\infty}2^{-lk}\sum _{\left\{ 2^{l-1}\le\mathrm d(n)< 2^l\right\}}1.
\end{split}
\end{equation}
By using Theorem 1.1 in \cite[p. 67]{model2}, we have
\begin{equation}\label{ft2}
\begin{split}
\sum _{\{0\le\mathrm d(n)<s\}}1&\le\begin{aligned}[t]&\left|\sum _{n\in\mathbb Z^d}\chi_{2(r+s)\mathcal D}(n)-\vol(2(r+s)\mathcal D)\right|+\Bigg|\sum _{n\in\mathbb Z^d}\chi_{2(r-s)\mathcal D}(n)\\
&-\vol(2(r-s)\mathcal D)\Bigg|+\vol(2(r+s)\mathcal D)-\vol(2(r-s)\mathcal D)\\
\end{aligned}\\
&\lesssim r^{(d-1)(1-1/\omega')}+r^{d-1}s\\
&\lesssim r^{(d-1)(1-1/\omega')}.
\end{split}
\end{equation}
On the other hand, notice that
\begin{align*}
\sum _{\left\{0\le \mathrm d(n)< 2^l\right\}}1\lesssim\left\{\begin{array}{ll}
r^{d-1}2^l&\textrm{if $2^l<2r$}, \\
2^{ld}&\textrm{if $2^l\ge2r$}.
\end{array}
\right.
\end{align*}
Then
\begin{equation}\label{ft3}
\epsilon^{-\tilde\gamma_k}s^{-k}t^k\sum _{\{ s\le\mathrm d(n)< 1 \}}1\lesssim\epsilon^{-\tilde\gamma_k}t^kr^{(d-1)(1+k/\omega')}
\end{equation}
with the choice $s=r^{-(d-1)/\omega'}$ and
\begin{equation}\label{ft4}
\sum_{l=1}^{+\infty}2^{-lk}\sum_{\{ 2^{l-1}\le \mathrm d(n)< 2^l\}}1
=\sum_{l=1,2^l\le 2r}^{+\infty}2^{-lk}r^{d-1}2^l+\sum_{l=1,2^l\ge 2r}^{+\infty}2^{-lk}2^{ld}
\lesssim r^{d-1}.
\end{equation}
Therefore, combing \eqref{ft1}--\eqref{ft4} yields
\begin{align*}
\sum _{n\in\mathbb Z^d}\hat f(n/t)&\lesssim \epsilon^{-\tilde\gamma_k}t^{\frac{d-1}{2}}r^{(d-1)\left(\frac{1}{2}-\frac{1}{\omega'}\right)}\left(1+t^{k-\frac{d-1}{2}}
r^{(d-1)\left(\frac{1}{2}+\frac{k+1}{\omega'}\right)}\right)\\
&\lesssim\epsilon^{-\tilde\gamma_k}t^{\frac{d-1}{2}}r^{(d-1)\left(\frac{1}{2}-\frac{1}{\omega'}\right)}\!\left(\!1+\!r^{-\alpha k+\alpha\frac{d-1}{2}+(d-1)\left(\frac{1}{2}+\frac{1}{\omega'}\right)+\frac{d-1}{\omega'}k}\right) \label{Y2}\\
&\lesssim \epsilon^{-\tilde\gamma_k}t^{\frac{d-1}{2}}r^{(d-1)\left(\frac{1}{2}-\frac{1}{\omega'}\right)},
\end{align*}
where the second inequality holds because of $t\lesssim r^{-\alpha}$. So we get
\begin{equation}\label{I1}
I_1=r^{d-1}t\sum _{n\in\mathbb Z^d}\hat f(n/t)\lesssim r^{d-1}t\epsilon^{-\tilde\gamma_k}t^{\frac{d-1}{2}}r^{(d-1)\left(\frac{1}{2}-\frac{1}{\omega'}\right)}.
\end{equation}
Hence if we take
\begin{equation*}
\epsilon=\left(t^{\frac{d-1}{2}}
r^{(d-1)\left(\frac{1}{2}-\frac{1}{\omega'}\right)}\right)^{1/\left(\Lambda+\tilde\gamma_k\right)},
\end{equation*}
combining \eqref{II}, \eqref{I2} and \eqref{I1} yields
\begin{equation*}
Y(r,t)\lesssim I_1+I_2+II\lesssim r^{d-1}t^{1+(d-1)\left(\frac{1}{2}
-\frac{1}{\alpha}\left(\frac{1}{2}-\frac{1}{\omega'}\right)\right)
\Lambda/\left(\Lambda+\tilde\gamma_k\right)},
\end{equation*}
where the implicit constant is independent of $r$ and $t$.
This finishes the proof.
\end{proof}



\section{Appendix}
\textbf{1. Proof of Lemma \ref{lemma1}.}

Given any $x'\in \partial\mathcal B$, we may assume $x'$ has a small neighborhood denoted by $U_{x'}$ on $\partial \mathcal B $ that can be given as a graph $x_d=f(x_1,\ldots,x_{d-1})$ and the $d^{th}$ component of $n(x')$ is positive, as all other cases can be proved similarly. Then
\begin{equation}\label{normal}
n(x)=\frac{\left(-\nabla f(x_1,\ldots,x_{d-1}),1\right)}{\left(1+\left|\nabla f(x_1,\ldots,x_{d-1})\right|^2\right)^{1/2}}.
\end{equation}
For any $x\in U_{x'}$, we can parameterize $x$ as
\begin{equation*}
\begin{split}
x=x'+\sum _{j=1}^{d}X_j\vec{\mathbf{t}}^j,
 \end{split}
\end{equation*}
 where $\vec {\mathbf{t}}^d=n(x')$ and
 \begin{equation*}
 \{\vec {\mathbf{t}}^1,\ldots,\vec {\mathbf{t}}^{d-1}\}=\left\{\left(t_1^1,\dots,t_d^1\right),\ldots,\left(t_1^{d-1},\dots,t_d^{d-1}\right)\right\}
  \end{equation*}
  is an orthonormal basis of the tangent plane of $\partial \mathcal B$ at $x'$ such that the basis $\{\vec {\mathbf{t}}^1,\ldots, \vec {\mathbf{t}}^d\}$ has the same orientation as
$\{\mathbf{e_1},\ldots,\mathbf{e_d}\}$. It is easy to get that
 \begin{equation}\label{determinant ti,j}
 \det\left(\textbf( t_j^i\textbf)_{i,j=1}^{d-1}\right)=\left(1+|\nabla f(x_1',\ldots,x_{d-1}')|^2\right)^{-1/2}.
\end{equation}

Let
\begin{equation*}
F(X)=x_d'+\sum_{j=1}^{d}X_j t_d^j- f\left(x_1'+\sum_{j=1}^{d}X_j t_1^j,\ldots,x_{d-1}'+\sum_{j=1}^{d}X_j t_{d-1}^j\right).
\end{equation*}
Then for any $1\le i\le d$,
\begin{equation}\label{ab value of normal}
 \partial_{X_i} F(\vec 0)=\left\{\begin{array}{ll}
 \left(1+|\nabla f(x_1',\ldots, x_{d-1}')|^2\right)^{1/2}\vec {\mathbf{t}}^d\cdot \vec {\mathbf{t}}^i=0&\text{if}\,  i<d,\\
  \left(1+|\nabla f(x_1',\ldots, x_{d-1}')|^2\right)^{1/2}\ne 0&\text{if}\,i=d.
 \end{array}
 \right.
\end{equation}
Note that $x_d=f(x_1,\ldots,x_{d-1})$ implies $F(X)=0$ and \eqref{ab value of normal} ensures that there is a small neighborhood of $\vec 0$ and a function $g$ defined in this neighborhood such that $g(\vec 0)=0$, $\nabla g(\vec 0)=\vec 0$ and
\begin{equation*}
F(X_1,\ldots,X_{d-1},g(X_1,\ldots,X_{d-1}))=0.
\end{equation*}
 Combining \eqref{determinant ti,j} and \eqref{ab value of normal} yields
\begin{align*}
\det\left(\nabla^2g(\vec{0})\right)&=\frac{\left(\det\left(\textbf( t_j^i\textbf)_{i,j=1}^{d-1}\right)\right)^2\det\left(\nabla^2 f(x_1',\ldots, x_{d-1}')\right)}{\left(1+|\nabla f(x_1',\ldots, x_{d-1}')|^2\right)^{(d-1)/2}}\\
&=\frac{\det\left(\nabla^2 f(x_1',\ldots, x_{d-1}')\right)}{\left(1+|\nabla f(x_1',\ldots, x_{d-1}')|^2\right)^{(d+1)/2}}.
\end{align*}
Hence
\begin{equation*}
K\left(x'\right)=\frac{\left|\det\left(\nabla^2 f(x_1',\ldots, x_{d-1}')\right)\right|}{\left(1+|\nabla f(x_1',\ldots, x_{d-1}')|^2\right)^{(d+1)/2}}.
\end{equation*}

If $K(x')\ne 0$. Let $\xi/|\xi|=n(x)$ and $\xi'/|\xi'|=n(x')$. Define
\begin{align*}
&\mathbf{G}(\xi_1,\ldots,\xi_d,x_1,\ldots,x_{d-1})\\
&=\left(\xi_1+\xi_d\partial_{x_1} f(x_1,\ldots,x_{d-1}) ,\ldots,\xi_{d-1}+\xi_d\partial_{ x_{d-1}} f(x_1,\ldots,x_{d-1})\right).
\end{align*}
Then
\begin{align}\label{implicit group}
\begin{split}
&|\det\left(D_{x} \mathbf{G}(\xi_1',\ldots,\xi_d',x_1',\ldots,x_{d-1}')\right)|\\
&=\xi_d'^{d-1}|\det\left(\nabla^2f(x_1',\ldots, x_{d-1}')\right)|\\
&=\xi_d'^{d-1}K\left(x'\right)\left(1+|\nabla f(x_1',\ldots, x_{d-1}')|^2\right)^{(d+1)/2}\\
&\ne0.
\end{split}
\end{align}
Note that equation \eqref{normal} implies
\begin{align*}
\mathbf{G}(\xi_1,\ldots,\xi_d,x_1,\ldots,x_{d-1})=\vec 0.
\end{align*}
The implicit function theorem and \eqref{implicit group} ensure that there is a small neighborhood of $\xi'/|\xi'|$ on $S^{d-1}$ denoted by $V_{\xi'/|\xi'|}$ and functions $h_1,\ldots,h_{d-1}$ defined in $V_{\xi'/|\xi'|}$ such that for any $\xi/|\xi|\in V_{\xi'/|\xi'|}$,
\begin{align*}
\mathbf{G}(\xi_1,\ldots,\xi_d,h_1(\xi_1,\ldots,\xi_d),\ldots,h_{d-1}(\xi_1,\ldots,\xi_d))=\vec 0
\end{align*}
and \eqref{inverse of Gauss map} and \eqref{inverse of Gauss map q} can be obtained directly by calculation.
This finishes the proof.

\textbf{2. Proof of Lemma \ref{relation}.}

For the domain $\mathcal D$, given any nonzero $\xi\in \mathbb R^d$, we may assume $\xi_i\ge 0$ for $1\le i\le d$ and $\xi_d/|\xi|\ge\varepsilon_0$, as all other cases can be proved similarly. Then there is a small neighbourhood of $x(\xi)$ on $\partial\mathcal D$ such that for any point $x$ in this neighborhood, the $d^{th}$ component
\begin{align*}
x_d&=\left(\left(1-\sum _{p=0}^{n-2}\left(\sum _{l=1+d_{p}}^{d_{p+1}} x_{l}^{\omega_{l}}
\right)^{m_{p+1}}\right)^{1/m_n}-\sum _{l=1+d_{n-1}}^{d-1} x_{l}^{\omega_{l}}\right)^{1/\omega_d}\\
&=:f(x_1,x_2,\cdots,x_{d-1})
\end{align*}
and $x_d\asymp 1$.
Since $\xi/|\xi|\asymp (-\nabla f(x_1,\cdots,x_{d-1}),1)$, we obtain
\begin{equation}\label{lablaf}
\xi_i/|\xi|\asymp \left(\sum _{l=1+d_{p(i)}}^{d_{p(i)+1}}\!x_l^{\omega_l}\right)^{m_{d,i}-1}x_i^{\omega_i-1}
\end{equation}
for $1\le i\le d-1$, which implies
\begin{equation}\label{xi}
\left(\xi_i/|\xi|\right)^{1/(\omega_i-1)}\lesssim x_i\lesssim \left(\xi_i/|\xi|\right)^{1/(m_{d,i}\omega_i-1)},
\end{equation}
where the implicit constants depend on the domain $\mathcal D$ and $\varepsilon_0$.

By a straightforward calculation we have
\begin{equation}\label{determinant}
\begin{split}
\left|\det(\nabla^2f)\right| \asymp \prod _{i=1}^{d-1}\left(\sum _{l=1+d_{p(i)}}^{d_{p(i)+1}}x_l^{\omega_l}\right)^{m_{d,i}-1}x_i^{\omega_i-2},
\end{split}
\end{equation}
where the implicit constants depend on the domain $\mathcal D$ and $\varepsilon_0$.
Combining \eqref{Gaussiancurvature}, \eqref{lablaf}--\eqref{determinant} yields
\begin{equation*}
\prod _{i=1 }^{d-1}\left(\xi_i/|\xi|\right)^{\frac{m_{d,i}\omega_i-2}{m_{d,i}\omega_i-1}}\le K(x(\xi))\lesssim \prod _{i=1}^{d-1}\left(\xi_i/|\xi|\right)^{\frac{\omega_i-2}{\omega_i-1}}.
\end{equation*}
This finishes the proof.

\textbf{3. Gaussian curvature of hypersurface $x(\xi)\cdot\xi=h$.}

Given any $\xi\in\{\xi\in \mathbb R^d\setminus \{\vec 0\}:x(\xi)\cdot\xi=h\}$, we may assume the $d^{th}$ component $x_d(\xi)\ge \varepsilon_0$, as all other cases can be proved similarly. Let $F(\xi)=x(\xi)\cdot\xi-h$. Then $\nabla F(\xi)= x(\xi)$ implies that there is a neighborhood of $\xi$ such that $x(\xi)\cdot\xi=h$ can be written as $\xi_d=g(\xi_1,\ldots,\xi_{d-1})$ by the implicit function theorem. Moreover,
\begin{equation*}
\nabla g(\xi_1,\ldots,\xi_{d-1})=\left(-x_1(\xi)/x_d(\xi),\ldots,-x_{d-1}(\xi)/x_d(\xi)\right)
 \end{equation*}
 and
\begin{equation}\label{h}
\begin{split}
&\det\left(\nabla^2g(\xi_1,\ldots,\xi_{d-1})\right)\\
&=(x_d(\xi))^{-2(d-1)}\det\left(\left(-x_d(\xi)\partial_{\xi_j} x_i(\xi)+x_i(\xi)\partial{\xi_j} x_d(\xi)\right)_{i,j=1}^{d-1}\right).
\end{split}
\end{equation}

Starting from the second row, add the $i^{th}$ row multiplied by $-x_{i-1}/x_{i}$ to the $(i-1)^{th}$ row for $i=2,\ldots,d-1$. Then
\begin{equation}\label{APmatrix3}
\begin{split}
&\det\left(\nabla^2g(\xi_1,\ldots,\xi_{d-1})\right)\\
&=(x_d(\xi))^{-(d-1)}\det\left(\left(-\frac{\partial x_i(\xi)}{\partial \xi_j}+\frac{x_i}{x_{i+1}}\frac{\partial x_{i+1}(\xi)}{\partial \xi_j}\right)_{i,j=1}^{d-1}\right).
\end{split}
\end{equation}
Since $x_d(\xi)\ge\varepsilon_0$, the boundary $\partial\mathcal D$ in a small neighborhood of $x(\xi)$ can be given as a graph $x_d=f(x_1,\ldots,x_{d-1})$ and by \eqref{normal}
\begin{equation*}
\frac{\partial x_d}{\partial \xi_j}=\sum _{l=1}^{d-1}-\xi_l\xi_d^{-1}\frac{\partial x_l}{\partial\xi_j}
\end{equation*}
for $1\le j\le d-1$. Plugging this equality into the $(d-1)^{th}$ row in \eqref{APmatrix3} and adding the $i^{th}$ row multiplied by
\begin{equation*}
\begin{split}
\left\{\begin{array}{ll}
-\frac{x_{d-1}\xi_1}{x_d\xi_d}&\text{if $i=1$,}\\
-\sum _{l=1}^{i-1}\frac{x_{d-1}\xi_lx_l}{x_d\xi_dx_i}-\frac{x_{d-1}\xi_i}{x_d\xi_d}&\text{if $2\le i\le d-2$}
\end{array}
\right.
\end{split}
\end{equation*}
 to the $(d-1)^{th}$ row yields that the element in $(d-1)^{th}$ row and $j^{th}$ column becomes
\begin{equation*}
\left(-1-\sum _{l=1}^{d-1}\frac{\xi_lx_l}{x_d\xi_d}
\right)\frac{\partial x_{d-1}(\xi)}{\partial \xi_j}=-\frac{h}{x_d\xi_d}\frac{\partial x_{d-1}(\xi)}{\partial \xi_j},
\end{equation*}
where we use $x(\xi)\cdot\xi=h$. So the right side of \eqref{APmatrix3} can be reduced to
\begin{equation*}
\begin{split}
\frac{(-1)^{d-1}h}{\xi_dx_d^{d}}\det\left(\left(\frac{\partial x_i}{\partial \xi_j}\right)_{i,j=1}^{d-1}\right)=\frac{x_d^{-d}h\det\left(\textbf(A_{ji}\textbf)_{i,j=1}^{d-1}\right)}{\xi_d^d\left(\det\left(\nabla^2f\right)\right)^{d-1}}
\!=\frac{x_d^{-d}h}{\xi_d^d\det\left(\nabla^2f\right)},
\end{split}
\end{equation*}
where the first equality is because of \eqref{inverse of Gauss map} and
\begin{equation*}
\det\left(\textbf( A_{ji}\textbf)_{i,j=1}^{d-1}\right)=\left(\det\left(\nabla^2f\right)\right)^{d-2}
\end{equation*}
since $\textbf( A_{ji}\textbf)_{i,j=1}^{d-1}$ is the adjoint matrix of $\nabla^2f$. Then
\begin{equation*}
\begin{split}
\widetilde K(\xi)&=\left(1+|\nabla g|^2\right)^{-(d+1)/2}|\det(\nabla^2g)|\\
&=|x(\xi)|^{-(d+1)}
\frac{x_dh}{\xi_d^dK(x(\xi))(1+|\nabla f|^2)^{(d+1)/2}}.
\end{split}
\end{equation*}
Since $ x_d\asymp1$ and $h=x(\xi)\cdot\xi\asymp|\xi|\asymp \xi_d$, we have
\begin{align*}
\widetilde K(\xi)\asymp (K(x(\xi)))^{-1}h^{-(d-1)}.
\end{align*}
 This finishes the proof.

\textbf{4. Counterexample.}

In this part, we give an example to show that the assumption of the range of $\alpha$ in Theorem \ref{variance} cannot be removed. Let
\begin{equation*}
\mathcal D=\{x\in\mathbb R^3 : x_1^6+x_2^6+x_3^{10}\le1\}
\end{equation*}
and $\mathcal D(r,t)=(r+t/2)\mathcal D\setminus (r-t/2)\mathcal D$ with $r=k\in \mathbb N$ and $t=Cr^{-2}$. In this case $\alpha=2<\max\{\alpha_1,\alpha_2,\alpha_3\}=4$ and
\begin{align*}
\vol(\mathcal D(r,t))=\vol(\mathcal D)\left(3C+2(t/2)^3\right).
\end{align*}

If $(0,k,x_3),(x_1,k,0)\in \partial ((k+t/2)\mathcal D)$ then
\begin{align*}
x_3=\left(\left((k+t/2)^6-k^6\right)(k+t/2)^4\right)^{1/10}
\ge \left(k^9t\right)^{1/10}
\end{align*}
and
\begin{align*}
x_1=\left((k+t/2)^6-k^6\right)^{1/6}\ge \left(k^5t\right)^{1/6}.
\end{align*}
For any $u\in[-1/2,1/2)\times[-t/2,0]\times[-1/2,1/2)$, if $k$ is large enough,
\begin{align*}
N_{\mathcal D(r,t)}(u)-\vol(\mathcal D(r,t))\gtrsim \left(k^9t\right)^{1/10}\left(k^5t\right)^{1/6}-\vol(\mathcal D(r,t))\gtrsim k^{6/5}.
\end{align*}
 Then
\begin{align*}
\int_{\mathbb{T}^3}\left|N_{\mathcal D(r,t)}(u)-\vol(\mathcal D(r,t))\right|^2\,\mathrm du\gtrsim k^{12/5}t/2\gtrsim k^{2/5},
\end{align*}
which means the variance is much larger than $\vol(\mathcal D(r,t))$.



\end{document}